\documentclass[11pt]{amsart}
\usepackage{graphicx}

\usepackage{amsmath, amsfonts, amssymb, amsthm, hyperref}
\usepackage[table]{xcolor}
\usepackage{fullpage}

\newtheorem{theorem}{Theorem}[section]

\newtheorem{definition}[theorem]{Definition}
\newtheorem{example}[theorem]{Example}
\newtheorem{lemma}[theorem]{Lemma}

\newtheorem{proposition}[theorem]{Proposition}

\newtheorem{remark}[theorem]{Remark}

\numberwithin{theorem}{section}
\numberwithin{equation}{section}

\def\l{\lambda}
\def\a{\alpha}

\begin{document}

\title{Analogues of Alder-Type Partition Inequalities for Fixed Perimeter Partitions}
\author{Ling Chen}
\address{Occidental College}
\email{lchen2@oxy.edu}
\author{Isabelle Hernandez}
\address{Oregon State University}
\email{hernisab@oergonstate.edu}
\author{Zain Shields}
\address{University of California Berkeley}
\email{zainshields@berkeley.edu}
\author{Holly Swisher}
\address{Oregon State University}
\email{swisherh@oergonstate.edu}

\thanks{This work was done during the Summer 2023 REU program in Number Theory at Oregon State University supported by NSF grant DMS-2101906}

\begin{abstract}
In a 2016 paper, Straub proved an analogue to Euler's partition identity for partitions with fixed perimeter.  Later, Fu and Tang provided a refinement and generalization of Straub's analogue to $d$-distinct partitions as well as a result related to the first Rogers-Ramanujan identity.  Motivated by Alder-type partition identities and their generalizations, we build on work of Fu and Tang to establish generalized Alder-type partition inequalities in a fixed perimeter setting, and notably, a reverse Alder-type inequality.  

\end{abstract}

\maketitle

\section{Introduction and Statement of Results}

Given a positive integer $n$, a \textit{partition} $\pi$ of size $n$ is a 
nonincreasing sequence of positive integers called \textit{parts} that sum to $n$.  One can visualize a partition $\pi$ with a \textit{Ferrers diagram}, in which each part $\pi_i$ is represented by a row of $\pi_i$ dots.  From top to bottom, the rows are arranged in nonincreasing order and left-justified.  For example, the following diagram represents the partition $7+2+1$.
$$\begin{matrix}
    \bullet & \bullet & \bullet & \bullet & \bullet & \bullet & \bullet
    \\\bullet & \bullet
    \\\bullet
\end{matrix}$$ 

For any positive integer $n$, we define the function $p(n)$ to count the total number of partitions of size $n$.  We will use $p(n\mid \text{condition})$ to denote the number of partitions of size $n$ satisfying the specified condition.  

The study of partitions of size $n$ is profoundly related to many areas of mathematics (see \cite{Andrews} for an introduction to this topic).  However, partitions are also studied and counted from perspectives other than size.

For any partition $\pi$, we let $\alpha(\pi)$ denote the largest part and $\lambda(\pi)$ to denote the number of parts. We define the \emph{perimeter} of a partition to be the largest hook length of $\pi$, denoted $\Gamma(\pi)$, which is also given by the number of dots traversing the top row and the left column of its Ferrers diagram.  In other words, $\Gamma(\pi)=\alpha(\pi)+\lambda(\pi)-1.$ As an example, we show below the Ferrers diagrams of two partitions with perimeter 4.

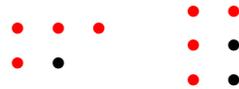
\begin{center}
    $\begin{matrix}
        \color{red}\bullet & \color{red}\bullet & \color{red}\bullet \\
        \color{red}\bullet & \bullet 
    \end{matrix}$
    \hspace{8mm}
    $\begin{matrix}
        \color{red}\bullet & \color{red}\bullet & \\
        \color{red}\bullet & \bullet \\
        \color{red}\bullet & \bullet
    \end{matrix}$
    \\
\end{center} 

As an analogue to the traditional partition function $p(n),$ we define $r(n)$ to count the number of partitions with fixed perimeter $n$. 

A fundamental partition theorem of Euler gives that for any positive integer $n$, the number of partitions of size $n$ into distinct parts is equal to the number of partitions of size $n$ into odd parts.  Namely,
\begin{align*}
    p(n\mid \text{distinct parts})=p(n\mid \text{odd parts}).
\end{align*}

In 2016, Straub proved the following analogue to Euler's Theorem for fixed perimeter partitions.

\begin{theorem}[Straub \cite{Straub}, 2016] \label{thm:Straub}
    The number of partitions with \textit{perimeter} $n$ into distinct parts is equal to the number of partitions of with \textit{perimeter} $n$ into odd parts.
\end{theorem}

In light of Straub's result it is natural to further investigate where partitions in terms of size may align or differ from partitions in terms of perimeter. \\

The celebrated Rogers-Ramanujan identities have partition-theoretic interpretations as
\begin{align*}
    p(n\mid \text{2-distinct parts}) = p(n\mid \text{parts} \equiv \pm 1 \!\!\!\!\! \pmod{5}), \\
    p(n\mid \text{2-distinct parts} \geq 2)=p(n\mid \text{parts} \equiv \pm 2 \!\!\!\!\! \pmod{5}).
\end{align*}
And relatedly, a theorem of Schur \cite{Schur} gives that the number of partitions of $n$ into parts that are 3-distinct is greater than or equal to the number of partitions of $n$ into parts congruent to $\pm 1$ modulo 6.  Namely,
\[
    p(n\mid \text{3-distinct parts}) \geq p(n\mid \text{parts} \equiv \pm 1 \!\!\!\!\! \pmod{6}).
\]
Motivated by these results, Alder \cite{alder} considered the following two functions,
\begin{align*}
q_d(n) &= p(n\mid \text{d-distinct parts}), \\
Q_d(n) &= p(n\mid \text{parts} \equiv \pm 1 \!\!\!\!\! \pmod{d+3}),
\end{align*}
and conjectured that $q_d(n) \geq Q_d(n)$ for all positive integers $d$ and $n$.  This conjecture was proved over the span of several decades by Andrews \cite{Andrews}, Yee \cite{Yee_04, Yee_08}, as well as Alfes, Jameson, and Lemke Oliver \cite{AJLO}.

In the style of Alder, Fu and Tang \cite{FuTang} generalized Theorem \ref{thm:Straub} by defining and establishing generating functions for
\begin{align*}
h_d(n) &= r(n\mid \text{parts are d-distinct}), \\
f_d(n) &=r(n\mid \text{parts are} \equiv 1 \!\!\!\!\!\pmod{d+1}),
\end{align*}
and proved the following theorem.
\begin{theorem}[Fu and Tang \cite{FuTang} 2018]
For all $n,d \in \mathbb{N}$, $$h_d(n)=f_d(n).$$
\end{theorem}
Fu and Tang further proved a refinement of their result for $d=1$.
\begin{theorem}[Fu and Tang \cite{FuTang} 2018] \label{thm:FT}
The number of partitions counted by $h_1(n)$ with exactly k parts is equal to the number of partitions counted by $f_1(n)$ with largest part of size $2k-1$ and both are enumerated by $\binom{n-k}{k-1}$.
\end{theorem}

There have been additional recent developments in the study of Alder-type inequalities.  In 2020, Kang and Park \cite{KP} generalized $q_d(n)$ and $Q_d(n)$ as follows 
\begin{align*}
q_d^{(2)}(n) &=p(n\mid \text{parts are d-distinct and} \geq 2), \\
Q_d^{(2)}(n) &=p(n\mid \text{parts are} \equiv \pm 2 \!\!\!\!\!\pmod{d+3}),
\end{align*}
in order to incorporate the second Rogers-Ramanujan Identity, which can be stated as $q_2^{(2)}(n)=Q_2^{(2)}(n)$.  Although it does not hold that $q_d^{(2)}(n) \geq Q_d^{(2)}(n)$ for all $d$ and $n$, Kang and Park \cite{KP}  introduced the modification 
$$Q_d^{(2,-)}(n)=p(n\mid \text{parts are} \equiv \pm 2 \!\!\!\! \pmod{d+3} \text{ excluding the part } d+1),$$ 
and conjectured that $q_d^{(2)}(n) \geq Q_d^{(2,-)}(n)$ for all positive integers $d$ and $n$. They prove their conjecture for $d=2^r-2$ when $r=1,2$ or $r\geq 4$.  

Later, Duncan, Khunger, Swisher and Tamura \cite{DKST} proved Kang and Park's conjecture for $d \geq 62$ using combinatorial and $q$-series techniques, and Sturman and Swisher \cite{Sturman, SS} prove all remaining cases except $d=3,4,5$ using asymptotic techniques.  Duncan, Khunger, Swisher, and Tamura also defined the generalizations 
\begin{align*}
q_d^{(a)}(n) & =p(n\mid \text{parts are d-distinct and} \geq a), \\
Q_d^{(a)}(n) & =p(n\mid \text{parts} \equiv \pm a \!\!\!\!\pmod{d+3}).
\end{align*}

Further work was done by Inagaki and Tamura \cite{IT} generalizing techniques to higher values of $a$. Armstrong, Ducasse, Meyer, and Swisher \cite{ADMS} also investigated general shifts of both $a$ and $d$. Kang and Kim \cite{K-K} compared $q_d^{(a)}(n)$ and $Q_d^{(a)}(n)$ along with the variations allowing the parameters of each term to vary independently and discovered results about general inequalities based on $a,d,$ and $n$. \\

The goal of this paper is to begin an investigation of further analogues of Alder-type inequalities in the fixed perimeter setting with regard to additional recent developments in this area.

 \subsection{Our Results}
As an attempt to investigate further Alder-type analogues in a fixed perimeter setting, we introduce the parameter $a$ to the functions $h_d(n)$ and $f_d(n)$.  For positive integers $n, a, d$, define
\begin{align} 
h_d^{(a)}(n) & =r(n\mid  \text{parts are d-distinct and } \geq a), \label{hdadef} \\
f_d^{(a)}(n) & = r(n\mid \text{parts are} \equiv a \!\!\!\!\!\pmod{d+1}), \label{fdadef}
\end{align}
and define the refinements
\begin{align*} 
h_d^{(a)}(\alpha,\lambda) & = h_d^{(a)}(\alpha,\lambda,n) \\
& = r(n\mid  \text{parts are $d$-distinct and } \geq a \text{ with largest part } \alpha \text{ and }\lambda \text{ parts}), \\
f_d^{(a)}(\alpha,\lambda) & = f_d^{(a)}(\alpha,\lambda,n) \\
& = r(n\mid  \text{parts are} \equiv a \!\!\!\!\!\pmod{d+1} \text{ with largest part } \alpha \text{ and }\lambda \text{ parts}).
\end{align*}
We further define $\mathcal{H}_d^{(a)}(n)$ to be the set of partitions counted by $h_d^{(a)}(n)$, and $\mathcal{F}_d^{(a)}(n)$ to be the set of partitions counted by $f_d^{(a)}(n)$.

The functions $h_d^{(a)}(n)$ and $f_d^{(a)}(n)$ generalize $h_d(n)$ and $f_d(n)$ in the same way that $q_d^{(a)}(n)$ and $Q_d^{(a)}(n)$ generalize $q_d(n)$ and $Q_d(n)$.  However, $f_d^{(a)}(n)$ is not a direct analogue of $Q_d^{(a)}(n)$.  As we are interested in exploring a more direct analogue of $Q_d^{(a)}(n)$ we are led to make the following definition.  For positive integers $n, a, d$, define 
\begin{equation}\label{ldadef}
\ell_d^{(a)}(n)= r(n\mid \text{parts are} \equiv \pm a \!\!\!\!\!\pmod{d+3}), 
\end{equation}
and the refinement
\begin{align*}
\ell_d^{(a)}(\alpha,\lambda) &= \ell_d^{(a)}(\alpha,\lambda,n) \\
&= r(n\mid  \text{parts are} \equiv \pm a \!\!\!\!\!\pmod{d+3} \text{with largest part } \alpha \text{ and }\lambda \text{ parts}).
\end{align*}
We further define $\mathcal{L}_d^{(a)}(n)$ to be the set of partitions counted by $\ell_d^{(a)}(n)$. \\

Our first result is the following analogue of Theorem \ref{thm:FT}.

\begin{theorem}\label{theorem:had=fad}
    For positive integers $d$, $n$, and $1 \leq a \leq d+1$, $$h_d^{(a)}(n)=f_d^{(a)}(n).$$ Moreover, the number of partitions counted by $h_d^{(a)}(n)$ with $\l$ parts equals the number of partitions counted by $f_d^{(a)}(n)$ with largest part $a+(d+1)(\l-1)$, namely, 
\[
h_d^{(a)}(\a, \l)=f_d^{(a)}(a+(d+1)(\l-1), \a-a-d(\l-1)+1).
\]
\end{theorem}

Our next result is a reverse Alder-type analogue which more closely aligns with the $q_d(n)$ and $Q_d(n)$ functions but with the inequality reversed.

\begin{theorem}\label{ReverseAlder}
    For positive integers $d$, $n$, and $a < \frac{d+3}{2}$, $$h_d^{(a)}(n) \leq \ell_d^{(a)}(n).$$
\end{theorem}

Lastly, we observe some shift inequalities for $h_d^{(a)}(n)$ and $\ell_d^{(a)}(n)$.  

\begin{proposition}\label{IneqChains}
For positive integers $d$, $n$, and $a$,
\begin{align*}   
h_d^{(a+1)}(n) & \leq h_d^{(a)}(n), \\
h_{d+1}^{(a)}(n) & \leq h_d^{(a)}(n), \\
h_d^{(a)}(n) & \leq h_d^{(a)}(n+1),
\end{align*}
and for positive integers $d$, $n$, and $a < \frac{d+3}{2}$,
\begin{align*}
\ell_d^{(a)}(n) &\leq \ell_d^{(a)}(n+1), \\
\ell_{d+1}^{(a)}(n) &\leq \ell_d^{(a)}(n).
\end{align*}
\end{proposition}

We note that the inequalities for $h_d^{(a)}(n)$ given in Proposition \ref{IneqChains} are direct analogues of properties of $q_d^{(a)}(n)$.  The fact that $h_d^{(a)}(n)=f_d^{(a)}(n)$ extends these results to $f_d^{(a)}(n)$ as well.  However, corresponding inequalities\footnote{A result of Xia \cite[Thm. 1]{Xia} does give that $Q_d^{(1)}(n)\geq Q_d^{(2)}(n)$ however.} do not hold for $Q_d^{(a)}(n)$ and in fact can all fail at the same time (for example when $d=6$, $a=2$, $n=18$).  Thus the inequalities given for $\ell_d^{(a)}(n)$ in Proposition \ref{IneqChains} are unique to the fixed perimeter setting.


We now outline the rest of the paper.  In Section \ref{genfunctionssection}, we extend Fu and Tang's findings to include the parameter $a$ and establish generating functions for $h_d^{(a)}(n)$ and $f_d^{(a)}(n)$. In Sections \ref{sec:duality} and \ref{reversealdersection} we prove Theorems \ref{theorem:had=fad} and \ref{ReverseAlder}, respectively.  In Section \ref{ineqchainssection} we prove Proposition \ref{IneqChains}.  Then, in Section \ref{furtherqssection}, we conclude with a brief discussion on some remaining questions. 

\section{Generating Functions}\label{genfunctionssection}

Observe that for any partition of perimeter $n$, its largest part $\alpha$ (or colloquially, its \emph{arm length}), its number of parts $\lambda$ (its \emph{leg length}), and perimeter $n$ are related by 
$$n=\alpha+\lambda-1,$$
and thus any two of these parameters will determine the other.  We will switch between parameter variables for convenience, i.e., rewriting $h_d^{(a)}(\alpha,\lambda,\alpha+\lambda-1)=h_d^{(a)}(n-\lambda+1,\lambda,n).$

For any positive integer $d$, Fu and Tang \cite[Thm. 2.15]{FuTang} obtained the following generating functions for $h_d^{(1)}(\alpha,\lambda)$ and $f_d^{(1)}(\alpha, \lambda)$, 
\begin{align}
 H_d^{(1)}(x,y,q) & =\sum_{\alpha=0}^{\infty}\sum_{\lambda=0}^{\infty}h_d^{(1)}(\alpha,\lambda)x^\alpha y^\lambda q^{\alpha + \lambda -1}=\frac{xyq}{1-(xq+x^dyq^{d+1})}, \label{Hd1}  \\
F_d^{(1)}(x,y,q) & =\sum_{\alpha=0}^{\infty}\sum_{\lambda =0}^{\infty}f_d^{(1)}(\alpha, \lambda)x^{\alpha} y^{\lambda}q^{\alpha + \lambda -1}=\frac{xyq}{1-(yq+x^{d+1}q^{d+1})}.      \label{Fd1}
\end{align}

We begin by reviewing their proof of this result.  Associate to each partition a word in $\{E,N\}$ corresponding to its Ferrers diagram in the following way.  Starting at the lower left corner, move along the lower and right outer boundary of the Ferrers diagram until arriving at the upper right corner.  For each step, catalog a move to the right with an $E$ (east) and a move up with an $N$ (north).  For example, the partition $2+2+1$   
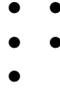
\[
\begin{matrix}
        \bullet  & \bullet \\
        \bullet  & \bullet \\
        \bullet & 
\end{matrix}
\]
is interpreted as the word $ENENN$.

As $x$ is keeping track of the size of the largest part, each $E$ will contribute an $x$.  Similarly, $y$ is keeping track of the number of parts, so each $N$ will contribute a $y$.  We have $q$ keeping track of the perimeter $\alpha + \lambda -1$.  Since every partition's word must start with an $E$ and end with an $N$, which together will account for only one square of the perimeter, each generating function will have an $xyq$ in the numerator.  For the rest of the perimeter, each $E$ and $N$ will each contribute one $q$.  Together, we have that the first and last $E,N$ contribute $xyq$, whereas each intermediate $E$ contributes a $xq$, and each intermediate $N$ contributes a $yq$. 

When restricting to $d$-distinct parts, $E$ movements are unrestricted, but any $N$ must be followed by at least $d$ copies of $E$.  Thus for each intermediate character we may choose either $xq$ or $x^dyq^{d+1}$, which yields \eqref{Hd1}.  Similarly, when restricting to parts congruent to $1$ modulo $d+1$, it is $N$ movements that are unrestricted, and any $E$ movements after the first must be done in increments of $d+1$.  Thus at each intermediate character we may choose either $yq$ or $x^{d+1}q^{d+1}$, which yields \eqref{Fd1}.

Both of these proofs for \eqref{Hd1} and \eqref{Fd1} can be easily modified to account for $a$ as in \eqref{hdadef} and \eqref{fdadef}.  Namely, in each case, whether parts are $\geq a$ or parts are $\equiv a \pmod{d+1}$, we must start with $a$ copies of $E$ and conclude with an $N$, which contributes $x^ayq^a$ in the numerator instead of $xyq$.  Then the choices for each intermediate character are the same as in the $a=1$ case.  From this we obtain the following generating functions for $h_d^{(a)}(\alpha,\lambda)$ and $f_d^{(a)}(\alpha, \lambda)$,

\begin{align}
H_d^{(a)}(x,y,q) & =\sum_{\alpha=0}^{\infty}\sum_{\lambda=0}^{\infty}h_d^{(a)}(\alpha,\lambda)x^\alpha y^\lambda q^{\alpha + \lambda -1}=\frac{x^ayq^a}{1-(xq+x^dyq^{d+1})}, \label{Hda} \\
F_d^{(a)}(x,y,q) & =\sum_{\alpha=0}^{\infty}\sum_{\lambda =0}^{\infty}f_d^{(a)}(\alpha, \lambda)x^{\alpha} y^{\lambda}q^{\alpha + \lambda -1}=\frac{x^ayq^a}{1-(yq+x^{d+1}q^{d+1})}.      \label{Fda}
\end{align}

We observe that setting $x=y=1$ in \eqref{Hda} and \eqref{Fda} immediately shows that
\[
\sum_{n=1}^\infty h_d^{(a)}(n)q^n=\frac{q^a}{1-(q+q^d)}=\sum_{n=1}^\infty f_d^{(a)}(n)q^n,
\]
which gives 
\begin{equation} \label{hd=fd}
h_d^{(a)}(n)=f_d^{(a)}(n),
\end{equation} 
the first statement in Theorem \ref{theorem:had=fad}.

\section{Refinement Formulas and Duality for $h_d^{(a)}(n)$ and $f_d^{(a)}(n)$} \label{sec:duality}

In this section we will prove Theorem \ref{theorem:had=fad}.  Expanding the generating functions \eqref{Hda} and \eqref{Fda} will lead us to refinement formulas for fixed arm and leg lengths in terms of binomial coefficients.  The following lemmas will be used in the proof of the duality part of Theorem \ref{theorem:had=fad}. 

\begin{lemma}\label{lem:Hbinom} 
Fix positive integers $d$, $a$, $\a$, and $\l$ such that $a \leq d+1$.  Then $h_d^{(a)}(\a,\l)$ is nonzero if and only if $\alpha \geq a+d(\l-1)$, and in that case,
\[
h_d^{(a)}(\a,\l)  = \binom{\a -a-(d-1)(\l-1)}{\l-1}. 
\]
\end{lemma}

\begin{proof}
Expanding \eqref{Hda} via geometric series and the binomial theorem, we obtain
\begin{equation} \label{Hexp}
\frac{x^ayq^a}{1-(xq+x^dyq^{d+1})} = x^ayq^a \sum_{j=0}^\infty (xq+x^dyq^{d+1})^j =  \sum_{j=0}^\infty \sum_{i=0}^j \binom{j}{i} x^{(d-1)i+j+a} y^{i+1} q^{di+j+a}.
\end{equation}
Given fixed $\alpha$, $\lambda$ we wish to find the coefficient of $x^\alpha y^\lambda q^{\alpha+\lambda-1}$.  The coefficient of $y^\l$ in \eqref{Hexp} is 
\[
\sum_{j=\l-1}^\infty \binom{j}{\l-1} x^{(d-1)(\l-1)+j+a} q^{d(\l-1)+j+a}.
\]
Thus when $j=\a-a-(d-1)(\l-1)$, we get the desired coefficient of $x^\alpha y^\lambda q^{\alpha+\lambda-1}$ in \eqref{Hexp}, namely, 
\[
\binom{\a -a-(d-1)(\l-1)}{\l-1}.
\]
We note that the hypothesis $\alpha \geq a+d(\l-1)$ is equivalent to the requirement that $j=\a-a-(d-1)(\l-1)\geq \l-1$.
\end{proof}

We next use a similar analysis on $f_d^{(a)}(\a,\l)$.

\begin{lemma}\label{lem:Fbinom} 
Fix positive integers $d$, $a$, $\a$, and $\l$ such that $a \leq d+1$.  Then $f_d^{(a)}(\a,\l)$ is nonzero if and only if $\alpha \geq a$ and $\alpha\equiv a \pmod{d+1}$, and in that case,
\[
f_d^{(a)}(\a,\l)  =\binom{\frac{\a-a}{d+1}+\l-1}{\l-1}.
\]
\end{lemma}

\begin{proof}
Expanding \eqref{Fda} via geometric series and the binomial theorem gives
\begin{align} \label{Fexp}
\frac{x^ayq^a}{1-(yq+x^{d+1}q^{d+1})} &=  x^ayq^a \sum_{j=0}^\infty (yq+x^{d+1}q^{d+1})^j \\
&=  \sum_{j=0}^\infty \sum_{i=0}^j \binom{j}{i} x^{(d+1)j-(d+1)i+a} y^{i+1} q^{(d+1)j-di+a}. \nonumber
\end{align}
The coefficient of $y^\l$ in \eqref{Fexp} is
\[
\sum_{j=\l-1}^\infty \binom{j}{i} x^{(d+1)j-(d+1)(\l-1)+a} q^{(d+1)j-d(\l-1)+a}.
\]
Thus when $(d+1)j=n-a+d(\l-1)$, we get the desired coefficient of $x^\alpha y^\lambda q^{\alpha+\lambda-1}$ in \eqref{Fexp}, namely, 
\[
\binom{\frac{\a-a}{d+1}+\l-1}{\l-1}.
\]
We note that the hypothesis $\alpha \geq a$ is equivalent to the requirement that $j=\frac{n-a+d(\l-1)}{d+1}\geq \l-1$.  Moreover for a partition to be counted by $f_d^{(a)}(\a,\l)$ it is required that $\alpha \equiv a \pmod{d+1}$.  Thus the top entry of the binomial coefficient is a nonnegative integer when $f_d^{(a)}(\a,\l)$ is nonzero.
\end{proof}

We now prove Theorem \ref{theorem:had=fad}.

\begin{proof}[Proof of Theorem \ref{theorem:had=fad}]
First, we have already shown the first statement, that $h_d^{(a)}(n)=f_d^{(a)}(n)$, in \eqref{hd=fd}.  Thus we next consider the refinements of $h_d^{(a)}(n)$ and $f_d^{(a)}(n)$ in terms of arm length $\a$ (or equivalently leg length $\l$).  Since $n=\a+\l-1$, it follows that $1\leq \a \leq n$ if and only if $1\leq \l \leq n$.  Thus we have
\begin{align*}
h_d^{(a)}(n)  &= \sum_{\a=1}^n h_d^{(a)}(\a,n-\a+1) = \sum_{\l=1}^n h_d^{(a)}(n-\l+1,\l), \\
f_d^{(a)}(n)  &= \sum_{\a=1}^n f_d^{(a)}(\a,n-\a+1) = \sum_{\l=1}^n f_d^{(a)}(n-\l+1,\l).
\end{align*}
However, not all terms in these sums are nonzero.  In particular, using Lemma \ref{lem:Hbinom}, we see that $h_d^{(a)}(\a,\l) \neq 0$ when $a+d(\l-1) \leq \a \leq n$, which is equivalent to $\l \leq \lfloor \frac{n-a}{d+1} \rfloor +1$.  Moreover, from Lemma \ref{lem:Fbinom} we have that $f_d^{(a)}(\a,\l) \neq 0$ when $a \leq \a \leq n$ and $\alpha\equiv a \pmod{d+1}$, which is equivalent to $\a = a+(d+1)(k-1)$ for integers $1\leq k \leq \lfloor \frac{n-a}{d+1} \rfloor +1$.  Namely,
\begin{equation} \label{h_refinement}
h_d^{(a)}(n)  = \sum_{\l=1}^{\lfloor \frac{n-a}{d+1} \rfloor +1}  h_d^{(a)}(n-\l+1,\l),  
\end{equation}
\begin{equation} \label{f_refinement}
f_d^{(a)}(n) =  \!\!\!\!\!\!\! \sum_{\substack{a\leq \a \leq n \\ \a \equiv a \!\!\!\! \pmod{d+1}}} \!\!\!\!\!  f_d^{(a)}(\a,n-\a+1) 
= \!\!\! \sum_{k=0}^{\lfloor \frac{n-a}{d+1} \rfloor} f_d^{(a)}(a + k(d+1), n-a-k(d+1)+1).
\end{equation}
Thus we see that we have the same number of terms in each sum.  We will now show that in fact these sums are merely permutations of each other.  Using the fact that $\binom{A}{B} = \binom{A}{A-B}$, we see from Lemma \ref{lem:Fbinom} that 
\begin{equation} \label{fdeval}
f_d^{(a)}(\a', \l') =  \binom{\frac{\a'-a}{d+1}+\l'-1}{\l'-1}  = \binom{\frac{\a'-a}{d+1}+\l'-1}{\frac{\a'-a}{d+1}}.
\end{equation}
Given $(\a, \l)$ with $n=\a+\l-1$, setting $\a' = a+(d+1)(\l-1)$ and $\l'= \a-a-d(\l-1)+1$ gives that $\a'+\l'-1 = \a+\l-1$, so the pair $(\a,\l)$ and $(\a',\l')$ will yield the same perimeter $n$.  From \eqref{fdeval} we obtain
\[
f_d^{(a)}(\a', \l') = \binom{\a-a-(d-1)(\l-1)}{\l-1} = h_d^{(a)}(\a, \l),
\]
which gives the desired duality.
\end{proof}

\begin{example}
As an example of the duality expressed in Theorem \ref{theorem:had=fad}, Figure \ref{fig:figure2} displays the fixed perimeter correspondence between $h_{1}^{(2)}(9)$ and $f_{1}^{(2)}(9)$. On the left, we show possible perimeters for partitions counted by $h_1^{(2)}(9)$, and on the right we show possible perimeters for partitions counted by $f_1^{(2)}(9).$  Leg lengths in the same row have the same number of associated partitions.
\end{example}

\scriptsize

\begin{center}

\begin{figure}[h!]
    $\begin{matrix}
        {\begin{matrix}
        \begin{array}{>{\columncolor{olive!20}}ccccccccc}
            \bullet & \bullet & \bullet & \bullet & \bullet & \bullet & \bullet & \bullet & \bullet \\
        \end{array}
        \end{matrix}} &

        {\begin{matrix}
        \begin{array}{>{\columncolor{olive!20}}cc}
            \bullet & \bullet \\
            \bullet \\
            \bullet \\
            \bullet \\
            \bullet \\
            \bullet \\
            \bullet \\
            \bullet \\
        \end{array}
        \end{matrix}} & \normalsize \text{1 partition} \scriptsize \\

    && \\
    && \\

        {\begin{matrix}
            \begin{array}{>{\columncolor{red!20}}cccccccc}
             \bullet & \bullet & \bullet & \bullet & \bullet & \bullet & \bullet & \bullet\\
             \bullet \\
             \end{array}
        \end{matrix}} &

        {\begin{matrix}
           \begin{array}{>{\columncolor{red!20}}cccc}
             \bullet & \bullet & \bullet & \bullet \\
             \bullet \\
             \bullet \\
             \bullet \\
             \bullet \\
             \bullet 
             \end{array}
        \end{matrix}} & \normalsize \text{6 partitions} \scriptsize \\

    && \\
    && \\

        {\begin{matrix}
        \begin{array}{>{\columncolor{blue!20}}ccccccc}
            \bullet & \bullet & \bullet & \bullet & \bullet & \bullet & \bullet\\
            \bullet \\
            \bullet 
        \end{array}
        \end{matrix}} &

        {\begin{matrix}
           \begin{array}{>{\columncolor{blue!20}}cccccc}
           \bullet & \bullet & \bullet & \bullet & \bullet & \bullet \\
            \bullet  \\
            \bullet  \\
            \bullet \\
            \end{array}
        \end{matrix}}& \normalsize \text{10 partitions} \scriptsize \\

&& \\
&& \\
      
        {\begin{matrix}
            \begin{array}{>{\columncolor{green!10}}cccccc}
            \bullet & \bullet & \bullet & \bullet & \bullet & \bullet\\
            \bullet \\
            \bullet \\
            \bullet 
            \end{array}
        \end{matrix}} &

        \begin{matrix}
        \begin{array}{>{\columncolor{green!10}}cccccccc}
            \bullet & \bullet & \bullet & \bullet & \bullet & \bullet & \bullet & \bullet \\
            \bullet 
        \end{array}
        \end{matrix} & \normalsize \text{4 partitions} \scriptsize \\
    \end{matrix}$  
\normalsize    
\caption{\label{fig:figure2}} Counting partitions with perimeter $9$ by arm and leg length.
\end{figure}

\normalsize

\end{center}

\normalsize 

\section{A reverse Alder-type analogue}\label{reversealdersection}

Here we compare $h_d^{(a)}(n)$ and $\ell_d^{(a)}(n)$ to investigate whether more directly analogous Alder-type analogues hold. 

\subsection{Degenerate Case} When $d$ is odd and $a=\frac{d+3}{2}$, then there is only one congruence class available for parts in partitions counted by $\ell_d^{(a)}(n)$.  Thus in this special case we obtain a direct Alder-type analogue. 

\begin{proposition} For positive integers $n,d$ with $d$ odd,
\[
h_d^{(\frac{d+3}{2})}(n) \geq \ell_d^{(\frac{d+3}{2})}(n).
\]
\end{proposition}

\begin{proof}
Since $d$ is odd, we have by definition and Theorem \ref{theorem:had=fad} that
\begin{equation}  \label{prop:lspecial}
\ell_d^{(\frac{d+3}{2})}(n)=f_{d+2}^{(\frac{d+3}{2})}(n) = h_{d+2}^{(\frac{d+3}{2})}(n).
\end{equation}
Moreover, by definition we see that $h_{d+2}^{(\frac{d+3}{2})}(n) \leq h_d^{(\frac{d+3}{2})}(n)$ since partitions that are $(d+2)$-distinct are also $d$-distinct. 
\end{proof}

\subsection{Nondegenerate Cases}
When $a < \frac{d+3}{2}$, $\ell_d^{(a)}(\alpha,\lambda)$ counts partitions with parts coming from two different classes modulo $d+3$.  In this subsection we will prove Theorem \ref{ReverseAlder}.  But first we establish preliminary inequalities and a refinement formula for $\ell_d^{(a)}(\alpha,\lambda)$.

Observe that for positive integers $x,y,z$ with $z\geq y$,
\[
\frac{(x+z)!}{x!}=\left(\frac{x+z}{z}\right) \left(\frac{x+z-1}{z-1}\right) ... \left(\frac{x+y+1}{y+1}\right) \frac{(x+y)!}{y!},
\]
and thus it follows that $\frac{(x+y)!}{y!} \leq \frac{(x+z)!}{z!}$, and
\begin{equation}\label{binom_ineq}
\binom{x+y}{y}\leq \binom{x+z}{z}.
\end{equation} 

\begin{lemma}\label{lem:fell_ineq}
Let $d,a,n$ be positive integers with $a < \frac{d+3}{2}$ and $n \geq a$.  Then,

\begin{align*}
0 \leq   \left \lfloor \frac12 \left \lfloor \frac{n-a}{d+1} \right \rfloor \right \rfloor   & \leq  \left \lfloor \frac{n-a}{d+3} \right \rfloor, \\
0 \leq  \left \lfloor \frac12 \left \lfloor \frac{n-a}{d+1} +1 \right \rfloor  \right \rfloor & \leq  \left \lfloor \frac{n+a}{d+3} \right \rfloor
\end{align*}

\end{lemma}

\begin{proof}
Since $d\geq 1$ and $n-a\geq 0$ it follows that $(d+3)(n-a) \leq 2(d+1)(n-a)$, and so
\[
\frac12 \left( \frac{n-a}{d+1} \right) \leq \frac{n-a}{d+3},
\]
which gives the first inequality. 

For the second, we first consider when $n<d+3-a$.  In this case $\lfloor \frac{n+a}{d+3} \rfloor =0$, so it suffices to show that $\lfloor \frac{n-a}{d+1} \rfloor =0$.  Since $2a+1\geq 3$, our assumption that $n<d+3-a$ implies that $n<d+a+1$, and thus $\lfloor \frac{n-a}{d+1} \rfloor =0$. 

Now let $n\geq d+3-a$.  Since $\lfloor x \rfloor \leq x$ and $x\leq y$ implies $\lfloor x \rfloor \leq \lfloor y \rfloor$, it suffices to show that 
\begin{equation} \label{eq:revised}
\frac12 \left( \frac{n-a}{d+1} +1\right)  \leq   \frac{n+a}{d+3}.
\end{equation}
If $d=1$, then the assumption $a < \frac{d+3}{2}$ forces $a=1$, and we obtain directly that both sides equal $\frac{n+1}{4}$.  When $d\geq 2$, inequality \eqref{eq:revised} is equivalent to $2(d+1)(n+a) \geq (d+3)(n-a+d+1)$, or in $n$,
\begin{equation} \label{eq:in_n}
n \geq \frac{(d+1)(d+3) -a(3d+5)}{d-1}.
\end{equation}
Since $a\geq 1$, writing $d+3=(3d+5)-(d-1)$ shows that $a((3d+5) - (d-1)) \geq 2(d+3)$, and further writing $2(d+3) = (d+1)(d+3) - (d-1)(d+3)$, yields 
\[
d+3-a \geq \frac{(d+1)(d+3) -a(3d+5)}{d-1}.
\]
Thus our assumption $n\geq d+3-a$ guarantees that \eqref{eq:in_n} and thus \eqref{eq:revised} is satisfied.
\end{proof}

We next prove a refinement formula for $\ell_d^{(a)}(\alpha,\lambda)$  in terms of binomial coefficients akin to Lemmas \ref{lem:Hbinom} and \ref{lem:Fbinom}. 

\begin{lemma}\label{prop:Lbincof}
Fix positive integers $d$, $a$, $\a$, and $\l$ such that $a < \frac{d+3}{2}$.  Then $\ell_d^{(a)}(\a,\l)$ is nonzero if and only if $\alpha \geq a$ and $\alpha\equiv \pm a \pmod{d+3}$, and in that case,
\[
\ell_d^{(a)}(\alpha, \lambda)=
\begin{cases} \displaystyle \binom{2(\frac{\alpha-a}{d+3})+\lambda-1}{\lambda-1}  & \text{ if } \alpha \equiv a \!\!\!\! \pmod{d+3}, \\ \displaystyle \binom{2(\frac{\alpha+a}{d+3})+\lambda-2}{\lambda-1}  & \text{ if } \alpha \equiv -a \!\!\!\! \pmod{d+3}. \end{cases}
\]
\end{lemma}

\begin{proof}
We can interpret this as a combinatorial problem.  The number of ways to choose $r$ elements from a set of size $k$ allowing repetitions is $\binom{k+r-1}{r}$ (a stars and bars problem, see \cite{Feller}).  The function $\ell_d^{(a)}(\alpha, \lambda)$ counts the number of partitions into exactly $\l$ parts from the set $\{x\in \mathbb{N} \mid x\equiv \pm a \pmod{d+3} \}$ such that the largest part is $\a$.  Since the part $\a$ is already known, that means we must count the number of ways to choose the other $\l-1$ parts.  Moreover, since $\a$ is the largest part, these parts must be chosen from $S=\{1\leq x\leq \a \mid x\equiv \pm a \pmod{d+3} \}$.  Thus we have that 
\[
\ell_d^{(a)}(\alpha, \lambda) = \binom{|S|+\l-2}{\l-1}.
\]  
The cardinality of $S$ depends on $\alpha$ modulo $d+3$.  In particular, 
\[
S=\{1\leq a+(d+3)k\leq \a \mid k\in \mathbb{Z}\} \, \sqcup \, \{1\leq -a+(d+3)k\leq \a \mid k\in \mathbb{Z}\}. 
\]
In the first subset, $0\leq k \leq \lfloor\frac{\alpha-a}{d+3}\rfloor$ and in the second $1\leq k \leq \lfloor\frac{\alpha+a}{d+3}\rfloor$, so 
\[
|S| = \begin{cases} 2(\frac{\alpha-a}{d+3})+1 & \text{ if } \alpha \equiv a \pmod{d+3}, \\  2(\frac{\alpha+a}{d+3}) & \text{ if } \alpha \equiv -a \pmod{d+3}, \end{cases}
\]
which gives us our desired result.  We note that it is clear from the definition that $\ell_d^{(a)}(\alpha, \lambda)$ is nonzero precisely when $\a$ is an allowable part. 
\end{proof}

\begin{remark}
We observe that this method will also give an alternate proof of Lemma \ref{lem:Fbinom}.  In that setting we wish to count the partitions with perimeter $\a+\l-1$, largest part $\a$, and parts coming from the set 
$S=\{1\leq a+(d+1)k\leq \a \mid k\in \mathbb{Z}\} = \left\{1\leq a+(d+1)k\leq \a \mid 0\leq k \leq \frac{\a-a}{d+1}\right\}$.  Thus 
\[
f_d^{(a)}(\a,\l)  = \binom{|S|+\l-2}{\l-1} = \binom{\frac{\a-a}{d+1}+\l-1}{\l-1}.
\] 
\end{remark}

We now consider the refinement of $\ell_d^{(a)}(n)$ in terms of arm length $\a$ (or leg length $\l$), as in the proof of Theorem \ref{theorem:had=fad}.  Since $n=\a+\l-1$, it follows that $1\leq \a \leq n$ if and only if $1\leq \l \leq n$.  We have
\[
\ell_d^{(a)}(n)  = \sum_{\a=1}^n \ell_d^{(a)}(\a,n-\a+1) = \sum_{\l=1}^n \ell_d^{(a)}(n-\l+1,\l),
\]
but not all terms in these sums are nonzero.  From Lemma \ref{prop:Lbincof}, we see that $\ell_d^{(a)}(\a,\l) \neq 0$ when $a \leq \a \leq n$ and $\alpha\equiv \pm a \pmod{d+3}$, which is equivalent to $\a$ satisfying either
\begin{align*}
\a & = a+(d+3)k, \text{ for } 0\leq k \leq \lfloor \frac{n-a}{d+3} \rfloor, \\
\a & = -a+(d+3)k, \text{ for } 1\leq k \leq \lfloor \frac{n+a}{d+3} \rfloor.
\end{align*}
Namely, for $n\geq a$,
\begin{multline} \label{ell_refinement}
\ell_d^{(a)}(n)  =   \sum_{\substack{a\leq \a \leq n \\ \a \equiv a \!\!\!\! \pmod{d+3}}} \!\!\!\!\!  \ell_d^{(a)}(\a,n-\a+1) +  \sum_{\substack{a\leq \a \leq n \\ \a \equiv -a \!\!\!\! \pmod{d+3}}} \!\!\!\!\!  \ell_d^{(a)}(\a,n-\a+1)\\
= \sum_{k=0}^{\lfloor \frac{n-a}{d+3} \rfloor} \ell_d^{(a)}(a + (d+3)k, n-a-(d+3)k+1) + \sum_{k=1}^{\lfloor \frac{n+a}{d+3} \rfloor} \ell_d^{(a)}(-a + (d+3)k, n+a-(d+3)k+1).
\end{multline}

\begin{remark}\label{rmk:arm_lengths}
Let $a < \frac{d+3}{2}$ and $n \geq a$.  We observe from \eqref{ell_refinement} that there are a total of $\left \lfloor \frac{n-a}{d+3} \right \rfloor + \left \lfloor \frac{n+a}{d+3} \right \rfloor +1$ possible arm lengths for partitions counted by $\ell_d^{(a)}(n)$.  Moreover from \eqref{f_refinement} we know there are $\left \lfloor \frac{n-a}{d+1} \right \rfloor + 1$ possible arm lengths for partitions counted by $f_d^{(a)}(n)$.  It follows that the number of possible arm lengths for partitions counted by $\ell_d^{(a)}(n)$ is equal to the sum of the number of possible arm lengths for partitions counted by $f_{d+2}^{(a)}(n)$ and $f_{d+2}^{(d+3-a)}(n)$, respectively.
\end{remark}

We are now able to prove Theorem \ref{ReverseAlder}, which gives a reverse Alder-type inequality for fixed perimeter partitions in the nondegenerate case.

\begin{proof} [Proof of Theorem \ref{ReverseAlder}]

We first observe that by definition, partitions counted by $\ell_d^{(a)}(n)$ and $h_d^{(a)}(n)$ must have parts of size at least $a$.  Thus when $n<a$, $\ell_d^{(a)}(n) = h_d^{(a)}(n) = 0$.  

For the remainder of the proof we assume $n\geq a$.  Since $d+1\geq \frac{d+3}{2}$ for all $d\geq 1$, from Theorem \ref{theorem:had=fad} it suffices to show that $\ell_d^{(a)}(n) \geq f_d^{(a)}(n)$ instead.

Applying Lemma \ref{prop:Lbincof} to \eqref{ell_refinement}, we obtain

\begin{equation} \label{ell_binom}
\ell_d^{(a)}(n)=\sum_{k=0}^{\lfloor \frac{n-a}{d+3} \rfloor} \binom{n-a-kd -k}{n-a-kd -3k} + \sum_{k=1}^{\lfloor \frac{n+a}{d+3}\rfloor}\binom{n+a-kd -k-1}{n+a-kd -3k}.
\end{equation}

Similarly, applying Lemma \ref{lem:Fbinom} to \eqref{f_refinement} gives that for $n\geq a$,
\[
 f_d^{(a)}(n)=\sum_{k=0}^{\lfloor \frac{n-a}{d+1} \rfloor}\binom{n-a-kd}{n-a-kd -k}.
\]
We split this into two sums based on whether $k$ is odd or even (and then reindex by $k$).  Using the fact that $\{0,1,\ldots,N\}$ contains $\lfloor \frac{N}{2} \rfloor +1$ evens and $\lfloor \frac{N+1}{2} \rfloor$ odds, we obtain for $n\geq a$,
\begin{equation} \label{f_binom}
f_d^{(a)}(n)=\sum_{k=0}^{\left \lfloor \frac12 \left \lfloor \frac{n-a}{d+1} \right \rfloor \right \rfloor}\binom{n-a-2kd}{n-a-2kd-2k}+\sum_{k=1}^{\left \lfloor \frac12  \left \lfloor \frac{n-a}{d+1} +1 \right \rfloor \right \rfloor} \binom{n-a-(2k-1)d}{n-a-(2k-1)d-(2k-1)}.
\end{equation}
From Lemma \ref{lem:fell_ineq}, we see that the first and second sums in \eqref{f_binom} contain no more terms than the first and second sums in \eqref{ell_binom}, respectively.  Thus we can compare each of the summands in \eqref{f_binom} directly to their counterparts in \eqref{ell_binom}.  For the first sums, setting $x=2k$, $y=n-a-2kd-2k$, and $z=n-a-kd-3k$ in \eqref{binom_ineq} gives the desired inequality.  For the second sums, setting $x=2k-1$, $y=n-a-(2k-1)d-(2k-1)$, and $z=n+a-kd-3k$ in \eqref{binom_ineq} gives the desired inequality.

\end{proof}

\begin{remark}
We note that Theorem \ref{ReverseAlder} can be proven by other comparisons of binomial coefficients as well (see \cite[Second pf. of Thm. 1.7]{CHS_proc}\footnote{Note that \cite[First pf. of Thm. 1.7]{CHS_proc} contains an error that is corrected in this paper.}).
\end{remark}

\section{Shift Inequalies}\label{ineqchainssection}

We now prove the shift inequalities for $h_d^{(a)}(n)$ and $\ell_d^{(a)}(n)$ in Proposition \ref{IneqChains}.

\begin{proof}[Proof of Proposition \ref{IneqChains}]
The inequalities for $h_d^{(a)}(n)$ follow directly in the same way as the corresponding inequalities for $q_d^{(a)}(n)$.  Namely, the first two are given by natural inclusions of the partitions counted by $h_{d}^{(a+1)}(n)$ and $h_{d+1}^{(a)}(n)$ into those counted by $h_{d}^{(a)}(n)$. The third is achieved by an injection from the partitions counted by $h_d^{(a)}(n)$ into those counted by $h_d^{(a)}(n+1)$ by increasing the largest part of any partition counted by $h_d^{(a)}(n)$ by $1$. 
        
The inequalities for $\ell_d^{(a)}(n)$ do not have direct analogues for $Q_d^{(a)}(n)$.  To show $\ell_d^{(a)}(n) \leq \ell_d^{(a)}(n+1)$, we define an injection on partitions counted by $\ell_d^{(a)}(n)$ by adding a part of size $a$.  This increases the perimeter by exactly $1$, and thus results in a partition counted by $\ell_d^{(a)}(n+1)$.  

To show $\ell_{d+1}^{(a)}(n) \leq \ell_d^{(a)}(n)$, we first observe from \eqref{ell_binom} that
\begin{align*}
\ell_d^{(a)}(n) & =\sum_{k=0}^{\lfloor \frac{n-a}{d+3} \rfloor} \binom{n-a-kd -k}{n-a-kd -3k} + \sum_{k=1}^{\lfloor \frac{n+a}{d+3}\rfloor}\binom{n+a-kd -k-1}{n+a-kd -3k}, \\
\ell_{d+1}^{(a)}(n) & =\sum_{k=0}^{\lfloor \frac{n-a}{d+4} \rfloor} \binom{n-a-kd -2k}{n-a-kd -4k} + \sum_{k=1}^{\lfloor \frac{n+a}{d+4}\rfloor}\binom{n+a-kd -2k-1}{n+a-kd -4k}.
\end{align*}
   
Given the shape of the terms, and since clearly $\left \lfloor \frac{n-a}{d+4} \right \rfloor \leq \left \lfloor \frac{n-a}{d+4} \right \rfloor$ and $\left \lfloor \frac{n+a}{d+4} \right \rfloor \leq \left \lfloor \frac{n+a}{d+4} \right \rfloor$, we can directly compare terms using \eqref{binom_ineq}.  Namely, letting $x=2k$, $y=n-a-kd-4k$, and $z=n-a-kd-3k$, \eqref{binom_ineq} yields that
$$\binom {n-a-kd-k}{n-a-kd-3k} \geq \binom {n-a-kd-2k}{n-a-kd-4k}.$$
And letting $x=2k-1$, $y=n+a-kd-4k$, and $z=n+a-kd-3k$, \eqref{binom_ineq} yields that
$$\binom {n+a-kd-k-1}{n+a-kd-3k} \geq \binom{n+a-kd-2k-1}{n+a-kd-4k}.$$
Thus $\ell_{d+1}^{(a)}(n) \leq \ell_d^{(a)}(n)$.
\end{proof}

\section{Conclusion and further questions}\label{furtherqssection}

It is interesting to observe the various ways that partitions with perimeter $n$ align with and differ from partitions with size $n$.

For example it would be interesting to explore the result of Xia \cite{Xia} in the fixed perimeter setting.  Or to what extent shift equalities in $a$ fail.  Computational evidence suggests that for a given $a$, $\ell_d^{(a)}(n) \leq \ell_d^{(a+1)}(n)$ for all but finite $n$. 

Furthermore, Kang and Kim \cite{K-K} have compared asymptotics of $q_d^{(a)}(n)$ with a generalized version of $Q_d^{(a)}(n)$ given by 
\[
Q_d^{(a,b)}(n) = p(n \mid \text{parts} \equiv  a \mbox{ or } b \!\!\!\!\pmod{d+3}).
\]
They showed that there is a specific tipping point $M_d$ such that when $d+3>M_d$, $q_d^{(a)}(n)-Q_d^{(a,b)}(n)$ grows to $\infty$, and that when $d+3\leq M_d$, $Q_d^{(a,b)}(n)-q_d^{(a)}(n)$ grows to $\infty$.  In light of this result, we define an analogous generalization of $\ell_d^{(a)}(n)$.

\begin{definition} For positive integers $a,b,d,$ and $n,$
    \begin{multline*}
    \ell_d^{(a,b)}(n)= r(n\mid  \text{with parts } \equiv a,b \!\!\!\!\pmod{d+3})\\= \sum_{k=0}^{\left \lfloor \frac{n-a}{d+3} \right \rfloor} \binom{n-q-kd-k}{n-a-kd-3k} + \sum_{k=1}^{\left \lfloor \frac{n-b}{d+3} \right \rfloor} \binom{n-b-kd-k-1}{n-b-kd-3k}.
    \end{multline*}
\end{definition}
It would be interesting to explore $\ell_d^{(a,b)}(n)$ in the context of Kang and Kim \cite{K-K}.


\end{document}